\documentclass{article}

\usepackage{amsmath,amsthm,amssymb,amsfonts}

\usepackage{verbatim}
\usepackage{graphicx}

\usepackage{stmaryrd} 
\usepackage{hyperref}

\usepackage[colorinlistoftodos]{todonotes}

\newtheorem{thm}{Theorem}

\newtheorem{lem}[thm]{Lemma}

\newtheorem{prop}[thm]{Proposition}

\newtheorem{assert}[thm]{Assertion}

\newtheorem{remarks}[thm]{Remark}

\newtheorem{definition}[thm]{Definition}

\newtheorem{exl}[thm]{Example}

\numberwithin{thm}{section}

\newcommand{\adj}{\leftrightarrow}

\newcommand{\adjeq}{\leftrightarroweq}

\DeclareMathOperator{\id}{id}
\DeclareMathOperator{\Fix}{Fix}

\def\Z{{\mathbb Z}}
\def\N{{\mathbb N}}

\def\R{{\mathbb R}}

\begin{document}
\title{Remarks on Fixed Point Assertions in Digital Topology, 9}
\author{Laurence Boxer
\thanks{Department of Computer and Information Sciences, Niagara University, NY 14109, USA
and  \newline
Department of Computer Science and Engineering, State University of New York at Buffalo \newline
email: boxer@niagara.edu
\newline
ORCID: 0000-0001-7905-9643
}
}

\date{ }
\maketitle

\begin{abstract}
We continue a discussion of
published assertions that are incorrect, incorrectly
proven, or trivial, in the theory
of fixed points in digital topology.


\end{abstract}

\section{Introduction}
Published assertions about fixed points in digital topology include
some that are beautiful and many that are incorrect, incorrectly
proven, or trivial. This paper continues the work
of~\cite{BxSt19, Bx19, Bx19-3, Bx20, Bx22, BxBad6, BxBad7, BxBad8}
in discussing flaws in papers that have come to our
attention since acceptance for publication of~\cite{BxBad8}.

We quote~\cite{BxBad8}:
\begin{quote}
... the notion of
a ``digital metric space" has led many authors to attempt,
in most cases either erroneously or trivially, to modify fixed
point results for Euclidean spaces to digital images. 
This notion contains roots of all the
flawed papers studied in the current paper.
See~\cite{Bx20} for discussion of why ``digital metric space"
does not seem a worthy topic of further research.
\end{quote}

\section{Preliminaries}
Much of the material in this section is quoted or
paraphrased from~\cite{Bx20}.

We use $\N$ to represent the natural numbers,
$\N^* = \N \cup \{0\}$, and
$\Z$ to represent the integers. 
The literature uses both $|X|$
and $\#X$ for the cardinality of~$X$. 

A {\em digital image} is a pair $(X,\kappa)$, where $X \subset \Z^n$ 
for some positive integer $n$, and $\kappa$ is an adjacency relation on $X$. 
Thus, a digital image is a graph.
In order to model the ``real world," we usually take $X$ to be finite,
although occasionally we consider
infinite digital images, e.g., for digital analogs of
covering spaces. The points of $X$ may be 
thought of as the ``black points" or foreground of a 
binary, monochrome ``digital picture," and the 
points of $\Z^n \setminus X$ as the ``white points"
or background of the digital picture.

\subsection{Adjacencies, 
continuity, fixed point}

In a digital image $(X,\kappa)$, if
$x,y \in X$, we use the notation
$x \adj_{\kappa}y$ to
mean $x$ and $y$ are $\kappa$-adjacent; we may write
$x \adj y$ when $\kappa$ can be understood. 
We write $x \adjeq_{\kappa}y$, or $x \adjeq y$
when $\kappa$ can be understood, to
mean 
$x \adj_{\kappa}y$ or $x=y$.

The most commonly used adjacencies in the study of digital images 
are the $c_u$ adjacencies. These are defined as follows.
\begin{definition}
\label{cu-adj-Def}
Let $X \subset \Z^n$. Let $u \in \Z$, $1 \le u \le n$. Let 
$x=(x_1, \ldots, x_n),~y=(y_1,\ldots,y_n) \in X$. Then $x \adj_{c_u} y$ if 
\begin{itemize}
    \item $x \neq y$,
    \item for at most $u$ distinct indices~$i$,
    $|x_i - y_i| = 1$, and
    \item for all indices $j$ such that $|x_j - y_j| \neq 1$ we have $x_j=y_j$.
\end{itemize}
\end{definition}

\begin{definition}
\label{path}
{\rm (See \cite{Khalimsky})} 
    Let $(X,\kappa)$ be a digital image. Let
    $x,y \in X$. Suppose there is a set
    $P = \{x_i\}_{i=0}^n \subset X$ such that
$x=x_0$, $x_i \adj_{\kappa} x_{i+1}$ for
$0 \le i < n$, and $x_n=y$. Then $P$ is a
{\em $\kappa$-path} (or just a {\em path}
when $\kappa$ is understood) in $X$ from $x$ to $y$,
and $n$ is the {\em length} of this path.
\end{definition}

\begin{definition}
{\rm \cite{Rosenfeld}}
A digital image $(X,\kappa)$ is
{\em $\kappa$-connected}, or just {\em connected} when
$\kappa$ is understood, if given $x,y \in X$ there
is a $\kappa$-path in $X$ from $x$ to $y$. The {\rm $\kappa$-component of~$x$ in~$X$} is the
maximal $\kappa$-connected subset
of~$X$ containing~$x$.
\end{definition}

\begin{definition}
{\rm \cite{Rosenfeld, Bx99}}
Let $(X,\kappa)$ and $(Y,\lambda)$ be digital
images. A function $f: X \to Y$ is 
{\em $(\kappa,\lambda)$-continuous}, or
{\em $\kappa$-continuous} if $(X,\kappa)=(Y,\lambda)$, or
{\em digitally continuous} when $\kappa$ and
$\lambda$ are understood, if for every
$\kappa$-connected subset $X'$ of $X$,
$f(X')$ is a $\lambda$-connected subset of $Y$.
\end{definition}

\begin{thm}
{\rm \cite{Bx99}}
A function $f: X \to Y$ between digital images
$(X,\kappa)$ and $(Y,\lambda)$ is
$(\kappa,\lambda)$-continuous if and only if for
every $x,y \in X$, if $x \adj_{\kappa} y$ then
$f(x) \adjeq_{\lambda} f(y)$.
\end{thm}

We use $\id_X$ to denote the identity function on $X$, 
and $C(X,\kappa)$ for the set of functions 
$f: X \to X$ that are $\kappa$-continuous.

A {\em fixed point} of a function $f: X \to X$ 
is a point $x \in X$ such that $f(x) = x$. We denote by
$\Fix(f)$ the set of fixed points of $f: X \to X$.

As a convenience, if $x$ is a point in the
domain of a function $f$, we will often
abbreviate ``$f(x)$" as ``$fx$".
Also, if $f: X \to X$, we use
``$f^n$" for the $n$-fold composition
\[ f^n = \overbrace{f \circ \ldots \circ f}^{n}
\]

\subsection{Digital metric spaces}
\label{DigMetSp}
A {\em digital metric space}~\cite{EgeKaraca15} is a triple
$(X,d,\kappa)$, where $(X,\kappa)$ is a digital image and $d$ is a metric on $X$. The
metric is usually taken to be the Euclidean
metric or some other $\ell_p$ metric; 
alternately, $d$ might be taken to be the
shortest path metric. These are defined
as follows.
\begin{itemize}
    \item Given 
          $x = (x_1, \ldots, x_n) \in \Z^n$,
          $y = (y_1, \ldots, y_n) \in \Z^n$,
          $p > 0$, $d$ is the $\ell_p$ metric
          if \[ d(x,y) =
          \left ( \sum_{i=1}^n
          \mid x_i - y_i \mid ^ p
          \right ) ^ {1/p}. \]
          Note the special cases: if $p=1$ we
          have the {\em Manhattan metric}; if
          $p=2$ we have the 
          {\em Euclidean metric}.
    \item \cite{ChartTian} If $(X,\kappa)$ is a 
          connected digital image, 
          $d$ is the {\em shortest path metric}
          if for $x,y \in X$, $d(x,y)$ is the 
          length of a shortest
          $\kappa$-path in $X$ from $x$ to $y$.
\end{itemize}

\begin{definition}
    \label{nbd}
    Given $x \in X$, where
    $(X,\kappa)$ is a digital
    image, the {\rm neighborhood
    of~$x$ of radius} $r \in \N^*$
    is
    \[ N_{\kappa}(x,r) = \{ y \in X \mid d(x,y) \le r \}
    \]
    where $d$ is the shortest path
    metric for the $\kappa$-component of~$x$ in~$X$.
\end{definition}

We say a metric space $(X,d)$ is {\em uniformly discrete}
if there exists $\varepsilon > 0$ such that
$x,y \in X$ and $d(x,y) < \varepsilon$ implies $x=y$.

\begin{remarks}
\label{unifDiscrete}
If $X$ is finite or  
\begin{itemize}
\item {\rm \cite{Bx19-3}}
$d$ is an $\ell_p$ metric, or
\item $(X,\kappa)$ is connected and $d$ is 
the shortest path metric,
\end{itemize}
then $(X,d)$ is uniformly discrete.

For an example of a digital metric space
that is not uniformly discrete, see
Example~2.10 of~{\rm \cite{Bx20}}.
\end{remarks}

We say a sequence $\{x_n\}_{n=0}^{\infty}$ is 
{\em eventually constant} if for some $m>0$, 
$n>m$ implies $x_n=x_m$.
The notions of convergent sequence and complete digital metric space are often trivial, 
e.g., if the digital image is uniformly 
discrete, as noted in the following, a minor 
generalization of results 
of~\cite{HanBan,BxSt19}.

\begin{prop}
\label{eventuallyConst}
{\rm \cite{Bx20}}
Let $(X,d)$ be a metric space. 
If $(X,d)$ is uniformly discrete,
then any Cauchy sequence in $X$
is eventually constant, and $(X,d)$ is a complete metric space.
\end{prop}

The {\em diameter} of a bounded
set $X \subset (\Z^n,d)$ is
\[ diam(X) = \max \{ d(x,y) \mid
   x,y \in X \}.
\]

\section{\cite{Han19}'s contraction maps}
There are in S.E. Han's paper~\cite{Han19} 
unnecessarily complex notations, unnecessary material, poor exposition,
inappropriate citations, 
undefined notions, and 
incorrect deductions;
indeed, most of what is introduced 
in~\cite{Han19} is incorrectly
stated, false or incorrectly ``proven",
or trivial. We discuss
shortcomings of~\cite{Han19}.

\subsection{\cite{Han19}'s fixed point property (FPP)}
A set $X$ with a notion of 
continuity for its self-maps has
the {\em fixed point property}
(FPP) if every continuous
$f: X \to X$ has a fixed point.

On page 237 of~\cite{Han19}, we have
``... it turns out that a (finite or infinite) digital
plane or a (finite or infinite) digital line $(X,k)$ 
in $\Z^2$ with $\#X \ge 2$ was proved not to have the FPP...." Aside from the question of what a finite
digital plane or a finite digital line might be,
a much stronger result is known: 
\begin{thm}
\label{FPPcharacterize}
{\rm \cite{BEKLL}}
A digital image
$(X,\kappa)$ has the FPP if and only if 
$\#X = 1$.
\end{thm}

Further, discussion of the
FPP was unnecessary as the
FPP does not contribute 
constructively to any
of the paper's new assertions.
In light of Theorem~\ref{FPPcharacterize},
we conclude that Remark~4.1(2) of~\cite{Han19},
claiming to relate the FPP to
a non-trivial property of a
digital image, is nonsense.

\subsection{\cite{Han19}'s surfaces}
On page 239 of~\cite{Han19}, 
we find several definitions
related to the notion of a
{\em digital closed surface}. 
As these do not lead to any
of the paper's assertions,
their purpose seems to be to
bloat the paper and
the author's citation count.

\subsection{\cite{Han19}'s definition of $k$-DC maps}
Definition~6 of~\cite{Han19} states the following.
\begin{definition}
    \label{Han-kDC}
    Let $f:(X,k) \to (X,k)$ be a self-map of a
    digital image in~$\Z^n$. Let $d$ be the
    Euclidean metric for~$\Z^n$. If there exists
    $\gamma \in [0,1)$ such that for all $x,y \in X$,
    $d(fx,fy) \le \gamma d(x,y)$ then $f$ is a
    $k-DC$-self-map.
    Besides, we say that $f$ has a digital version of the Banach principle (DBP for short).
\end{definition}

We note the following.
\begin{itemize}
    \item \cite{Han19} attributes Definition~\ref{Han-kDC} to Han's own paper~\cite{HanBan}. 
    In fact,~\cite{HanBan} correctly
    attributes this definition to~\cite{EgeKaraca-Ban}.
    Thus Han appears to be 
    claiming credit for introducing 
    a notion for which he
    knows credit
    belongs to others.
    \item \cite{Han19} never tells us what ``DC" 
    represents. According to~\cite{EgeKaraca-Ban}, it
    abbreviates ``digital contraction".
\item The Banach Principle is generally not part of the definition of
a contraction map.
Rather, the digital Banach 
Principle is the following.
\begin{thm}
\label{DBPthm}
    {\rm \cite{EgeKaraca-Ban}}
    Let $(X, d, \kappa)$ be a complete digital metric space, where
$d$ is the Euclidean metric in $\Z^n$. Let $f : X \to X$ be a digital contraction map.
Then $f$ has a unique fixed point.
\end{thm}
\item Han attributes
Theorem~\ref{DBPthm} to 
both~\cite{EgeKaraca-Ban}
and to his own paper~\cite{HanBan}.
Since~\cite{HanBan} does not give a new proof of this
theorem, but merely (properly) 
cites~\cite{EgeKaraca-Ban} 
for it, it appears Han is again
claiming credit in~\cite{Han19}
that he knows belongs to others.
\end{itemize}

\subsection{\cite{Han19}'s Proposition 3.6}
\label{discontSec}
This ``proposition" claims that
a digital contraction map is
digitally continuous. A
counterexample to this claim
is found in Example~4.1
of~\cite{BxSt19}, parts of which we
restate for use in other sections 
of this paper.
\begin{exl}
   {\rm \cite{BxSt19}}
   \label{counterHan5.1}
    Let $X=\{x_0,x_1,x_2\} \subset \Z^5$,
    where
    \[ x_0=(0,0,0,0,0),~~~~~
    x_1=(2,0,0,0,0),~~~~~
    x_2=(1,1,1,1,1)
    \]
    \begin{enumerate}
        \item If $d$ is
    the Euclidean metric,
    \[ d(x_0,x_1) = 2 < \sqrt{5} = d(x_0,x_2) = d(x_1,x_2).\]
    \item
    \[ x_0 \adj_{c_5} x_2
    \adj_{c_5} x_1,
    ~~~~~
    \mbox{so $X$ is $c_5$-connected.}
    \]
    \item  Let $f: X \to X$
    be given by
    \[fx_0 = x_0 = fx_1,~~~~~fx_2 = x_1.\]
    Then $i \neq j$ implies
    \begin{equation}
    \label{BxStExlIneq}
    d(fx_i,fx_j) \le (2/\sqrt{5})d(x_i,x_j), 
    \end{equation}
    so~$f$ is
    a digital contraction map.
    \item $f$ is not $c_5$-continuous, 
    since $fx_0$ and
    $fx_2$ are neither equal nor
    $c_5$-adjacent.
    \end{enumerate}
\end{exl}

\subsection{\cite{Han19}'s (4.1)}
Statement~(4.1) of~\cite{Han19} claims the following.
\begin{assert}
\label{Han(4.1)}
    For any digital contraction map
    $f: (X,\kappa) \to (X,\kappa)$ and $x' \in \Fix(f)$,
    there is a positive integer~$n$ such that
    \[ X \supset f(X) \supset f^2(X) \supset \ldots 
     \supset f^n(X) = \{x'\}.
    \]
\end{assert}

Neither proof nor citation
is offered in~\cite{Han19}
for the claim that
$f^{i-1}(X) \supset f^i(X)$.
Thus, this claim is unproven.

Further, the assertion that
$f^n(X) = \{x'\}$ for some~$n$ 
is not generally true, as 
shown by the following.

\begin{exl}
Let $X = \{2^n\}_{n=0}^{\infty}$. Let 
$f: (X,c_1) \to (X,c_1)$
be defined by
\[ f(x) = \left \{ \begin{array}{ll}
    1 & \mbox{if } x = 1; \\
    x/2 &  \mbox{if } x > 1.
\end{array}
\right .
\]
Then for $n > 0$,
\[ d(f(1),f(2^n)) = d(1, 2^{n-1}) < 
1/2 \, \cdot d(1,2^n)
\]
and for $1 \le m \le n$,
\[ d(f(2^m), f(2^n)) = 2^{n-1} - 2^{m-1} =
   1/2 \cdot (2^n - 2^m) = 1/2 \, \cdot d(2^m,2^n).
\]
So $f$ is a contraction map 
with $1 \in \Fix(X)$. 
But $f^n(X) = X$ for all $n \in \N$.
\end{exl}

We can add an assumption of
finiteness to the hypotheses 
of Assertion~\ref{Han(4.1)}, and
use weaker assumptions about
the maps studied, to get
a conclusion that is valid, 
though not as strong as that
claimed by Assertion~\ref{Han(4.1)}.

\begin{thm}
    \label{Han(4.1)corrected}
    Let $(X, d)$ be a 
    finite metric space.
    Let $f: X \to X$ be a map 
    such that 
    \begin{equation}
    \label{distDecrease}
        \mbox{for $x,y \in X$,
    $x \neq y$, we have } 
    d(fx,fy) < d(x,y).
    \end{equation} 
    Then $f$ has a unique fixed point~$x'$,
    and there exists~$n \in \N$ 
    such that $f^n(X) = \{x'\}$.
\end{thm}

\begin{proof}
    Let $f^0 = \id_X$.
    Since $X$ is finite,~(\ref{distDecrease})
    implies~$f$ is a digital contraction map. By
    Theorem~\ref{DBPthm},
    $f$ has a unique fixed point~$x'$.

From~(\ref{distDecrease}), we have that
    \[ \mbox{if }~~~ \#f^{i-1}(X) > 1
    ~~~\mbox{ then }~~~
    diam(f^i(X)) < diam(f^{i-1}(X)).
    \]
    Since $diam(f^i(X)) \in
    \{d(x,y) \mid x,y \in X \}$,
    a set with finitely many members, it follows that
    for some~$n \in \N$, $diam(f^n(X)) =0$, so
    $\#f^n(X) = 1$, and
since $f^n(x')=x'$, it follows that
    $f^n(X) = \{ x' \}$.
\end{proof}

\subsection{\cite{Han19}'s complexity measure}
Definition~7 of~\cite{Han19} gives the following measure of 
complexity of digital images. 
\begin{definition}
    \label{HanComplexity}
    {\rm \cite{Han19}}
    Let $(X,d,\kappa)$ be a digital metric space,
    where $X \subset \Z^n$ and $d$ is the 
    Euclidean metric. Then
    \[ C^{\sharp}(X,\kappa) =
    \max \left \{
    \begin{array}{l} m \in \N \mid \mbox{ for some $\kappa-$ contraction
       mapping }
    f: X \to X \mbox{,} \\
     |f^{m-1}(X)| > 1
       \mbox{ and } |f^m(X)| = 1
    \end {array} \right \}
    \]
\end{definition}

Theorem~6.1 of~\cite{Han19} shows that 
$C^{\sharp}(X,\kappa)$ is not a digital 
topological invariant, i.e., there are
isomorphic digital images
$(X,\kappa)$ and $(Y,\lambda)$ such that 
$C^{\sharp}(X,\kappa) \neq C^{\sharp}(Y,\lambda)$.
This seems more evidence of the artificial nature
of the notion of a digital metric space.

\subsection{~\cite{Han19}'s finiteness requirement}
Just before Remark~4.1 of~\cite{Han19}, we find
``... in digital geometry we deal with only a 
finite digital image ...." This is an odd statement
from the author who introduced infinite digital
images for covering spaces~\cite{HanNon}.

\subsection{\cite{Han19}'s Lemma 4.2}
What Han calls the
$2n$ adjacency in
$\Z^n$ is what we
call the $c_1$ adjacency.
The following is
a paraphrased version
of \cite{Han19}'s Lemma~4.2.

\begin{lem}
    \label{Han4.2}
    Let $(X,c_1)$ be
    a digital simple
    closed curve.
    Let $f \in C(X,c_1)$ be a
    digital contraction map.
    Then $f(X) \subset N_{c_1}(x',1)$, where
    $x' \in \Fix(f)$.
\end{lem}

The following result contains
Lemma~\ref{Han4.2}.
It appeared as Theorem~4.7(1) of
Han's own paper
\cite{HanBan} and, slightly 
generalized, as Theorem~4.2
of~\cite{BxSt19}.

\begin{thm}
Let $(X,d,c_1)$ be a
connected digital metric space, where
$d$ is any $\ell_p$ metric on $\Z^n$.
Let $f: X \to X$ be
a digital contraction
map. Then $f$ is a
constant function.
\end{thm}

\subsection{\cite{Han19}'s Lemma~4.3}
 \begin{figure}
     \centering
     \includegraphics{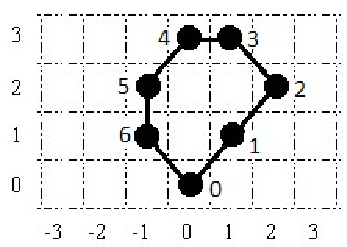}
     \caption{A digital simple closed curve
     $(X,c_2)$ of 7 points - 
     a counterexample to Han's claim that 
     $\ell$ must be even, where
     $\ell = \#S_k^{n,\ell}$. 
     $X = \{x_i\}_{i=0}^6$ with
     each vertex of the graph 
     marked by its index;
     e.g., $x_0=(0,0)$, $x_1=(1,1)$, etc.}
     \label{fig:7ptSCC}
 \end{figure}

Han uses the symbol $S_k^{n,\ell}$ for a digital
simple closed curve in~$\Z^n$ of~$\ell$ points, using
$k$-adjacency. Between the 
statement and the ``proof" 
of Lemma~4.3
of~\cite{Han19}, we find
``... the number~$\ell$ of~$S_k^{n,\ell}$ ... is an
even number....", an assertion he attributes to his
paper~\cite{HanNon}. Note the following.
\begin{itemize}
    \item \cite{HanNon} contains no such assertion.
    \item Figure~\ref{fig:7ptSCC} shows a counterexample
         to this assertion.
\end{itemize}

We also find, in the ``proof"
of ``Lemma"~4.3, toward the
bottom of page~243, the
self-contradictory ``... 
map~$f$ need not be a constant 
map but have [{\em sic}] the property $Im (f) = \{ x'\}$...."

\subsection{\cite{Han19}'s Theorem 4.4}
The following is a paraphrased 
version of~\cite{Han19}'s Theorem~4.4.
\begin{assert}
    \label{Han4.4}
   If $(X,c_u)$ is a digital simple
   closed curve for
   which every digital contraction $f$
   (using the Euclidean metric) that has a fixed
   point $x'$ satisfies
   $f(X) \subset N_{c_u}(x',1)$,
   then $C^{\sharp}(X,c_u) \le 2$.
\end{assert}

\begin{itemize}
\item We show below that
the claim appearing at the bottom of page~244 
of~\cite{Han19}, that this
assertion derives from
~\cite{Han19}'s Lemmas~4.2
and~4.3, is unnecessary.
\item There are errors in the
argument Han offers for its 
proof. We indicate the errors
and give a corrected proof of a somewhat
weaker assertion, as follows.
\end{itemize}

\begin{thm}
    \label{correctHan4.4}
       If $(X,c_u)$ is a digital simple
   closed curve for
   which every digital contraction $f$
   (using the Euclidean metric) that has a fixed
   point $x'$ satisfies
   $f(X) \subset N_{c_u}(x',1)$,
   then $C^{\sharp}(X,c_u) \le 3$.
\end{thm}

\begin{proof}
    We modify Han's
    argument.

    If for every such~$f$ we
    have $f(X) = \{x'\}$ then
    $C^{\sharp}(X,c_u)=1$.

    Otherwise, for some 
    contraction mapping~$f$
    on~$X$ we
    have $f(X) \neq \{x'\}$ but $f(X) \subset N_{c_u}(x',1)$. Here,
    Han claims that
    $\#N_{c_u}(x',1) \le 2$
    and that this yields the
    desired conclusion. 
    The correct statement is
    $\#N_{c_u}(x',1) = 3$
    (e.g., in Figure~\ref{fig:7ptSCC}, 
    $N_{c_2}(x_4,1)=\{x_3,x_4,x_5\}$).
    Neither Han's erroneous claim nor this corrected
    statement yields Han's claimed conclusion. 

However, for
    $fx \adj_{c_2} x'$, $d(fx,x') \in \{1, \sqrt{2} \}$.
    Since $f$ is a contraction map and $x'$
    is a fixed point of~$f$:
\begin{itemize}
    \item For each $fx \adj_{c_2} x'$ such that
          $d(fx,x')=1$, we have $d(f^2x,x') = 0$.
    \item For each $fx \adj_{c_2} x'$ such that
          $d(fx,x')=\sqrt{2}$, we have $d(f^2x,x') = 0$
          or $d(f^2x,x') = 1$. In the latter case,
          $d(f^3x,x')=0$.
\end{itemize}
The assertion follows.
\end{proof}

\subsection{\cite{Han19}'s uniform $k$-connectedness}
\label{unifConnected}
The following paraphrases Definition~8 of~\cite{Han19}.

\begin{definition}
    {\rm \cite{Han19}}.
    \label{unifConnected}
    A $c_u$-connected digital
    image $X$ is {\em uniformly $c_u$-connected}
    if $x \adj_{c_u} y$ in~$X$
    implies $d(x,y) = \sqrt{u}$.
\end{definition}

The following shows the triviality of
this concept in its most natural
setting.

\begin{prop}
    \label{unifConnectedTriv}
    Let $(X,c_u)$ be uniformly 
    $c_u$-connected. Let~$d$ be
    the Euclidean metric on~$X$ and let
    $f: X \to X$ be a $c_u$-continuous
    digital contraction map. Then~$f$
    is a constant map.
\end{prop}

\begin{proof}
    Let $x \adj_{c_u} y$ in~$X$.
    Then $fx \adjeq_{c_u} fy$, so
    $d(fx,fy) \in \{0,\sqrt{u} \}$, and
    \[ d(fx,fy) < d(x,y) = \sqrt{u}.
    \]
    Hence $d(fx,fy) = 0$, i.e., $fx=fy$.
    The assertion follows.
\end{proof}

The following is stated as
``Theorem"~5.1 of~\cite{Han19}.

\begin{assert}
    \label{Han5.1}
    Assume a digital image
    $(X,k)$, $X \subset \Z^n$
    and $|X| \ge 2$.
    If $(X,k)$ is uniformly 
    $k$-connected then 
    $C^{\sharp}(X,k) = 1$.
\end{assert}

We noted in
Example~\ref{discontSec}(3--4) that
a digital contraction
map need not be
digitally continuous.
The following is a
counterexample to
Assertion~\ref{Han5.1}.

\begin{exl}
  Let $X$ and~$f$ be as in
    Example~\ref{counterHan5.1}.
    $X$ is easily seen to be
    uniformly $c_5$-connected.
We saw that~$f$ is
    a digital contraction map.
    However, $\#f(X) = 2$.
    Hence $\#C^{\sharp}(X,c_5) > 1$.
\end{exl}

\begin{figure}
    \centering
    \includegraphics{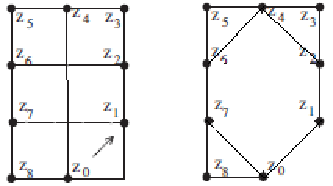}
    \caption{Left: a copy of~\cite{Han19}'s Figure~4(c), incorrectly 
    claimed to show 8-adjacency, i.e., $c_2$ 
    adjacency. Right: a correct showing of $c_2$ 
    adjacency for this set of vertices.}
    \label{fig:ForHanRemark5.2}
\end{figure}
In the following, let 
$X = \{x_i\}_{i=0}^{\ell-1}$ be a
digital simple closed curve,
where the members of~$X$ are
indexed circularly.
We paraphrase ``Theorem"~5.3 
of~\cite{Han19}.

\begin{assert}
\label{Han5.3}
If $(X,c_u)$ is a digital
simple closed curve such that
$C^{\sharp}(X,c_u)=1$
then $(X,c_u)$ is
uniformly $u$-connected.
\end{assert}

Note 
$\{x_0, x_2\} \subset N_{c_u}(X,x_1)$.
Note also that in the
$c_u$ adjacency for $\Z^n$
and the Euclidean metric,
$\sqrt{u}$ is the maximum
distance between adjacent
points.

Han's argument on behalf of
this assertion includes the
following.
Suppose, for a contradiction, 
that $(X,c_u)$ is not uniformly
$c_u$-connected. Then, 
renumbering if necessary,
\[ \sqrt{u} \ge d(x_1,x_2) \ge 
d(x_2,x_3) \ge 1.
\]
Construct $f: X \to X$ such that
\[ f(x_1) = x_2,~~~f(x_2)=f(x_3)=x_3,
\]
and $f(X) \setminus \{x_1, x_2,x_3\}$ 
has the digital
Banach property. Han then
claims that~$f$ is a digital
contraction map. But this
claim is unproven, since,
e.g., we don't know if
$d(fx_0,fx_1)$ is
less than $d(x_0,x_1)$.
Therefore, Assertion~\ref{Han5.3} is
unproven.

\subsection{\cite{Han19}'s
strictly $k$-connectedness}
Definition~9 of~\cite{Han19}
depends ``the
Euclidean distance between"
subsets of~$\Z^n$. However,
Han doesn't tell us what this
distance is. It might be the
Hausdorff distance~\cite{Nadler} based on the Euclidean metric, or 
perhaps Borsuk's metric of 
continuity~\cite{Borsuk}
based on the Euclidean metric,
or any of a variety of other
metrics that have been designed
for such a purpose. Therefore,
Han's ``Definition"~9 is undefined, and the
subsequent assertions that
depend on it (``Theorem"~5.5;
``Example"~5.6; and ``Example"~5.7, for which no
argument is supplied on behalf
of any of its assertions) are
unproven.

\section{\cite{NawazEtal}'s stability
assertions}
The paper~\cite{NawazEtal} claims to
introduce results concerning Hyers-Ulam 
stability~\cite{NawazEtal,Ulam} 
that are related to fixed points 
in digital metric spaces. In fact, the
paper fails to live up to this claim.

Among the shortcomings of~\cite{NawazEtal}:
its Theorem~1.11 is stated with
neither proof nor citation.

We consider \cite{NawazEtal}'s ``Lemma" 2.1 and dependent assertions.

A function $f: X \to Y$ is 
{\em additive} if 
$f(x+y) = f(x) + f(y)$.

``Lemma" 2.1 of~\cite{NawazEtal} states
the following, in which we understand
$(X,d,\kappa)$ and $(G,d,\kappa)$ to
be digital images, with~$d$ as the
Euclidean metric.

\begin{assert}
    \label{Nawaz2.1}
    Assume $f: (X,d,\kappa) \to (G,d,\kappa)$ to be a mapping
    satisfying $f(0)=0$ also {\rm [{\em sic}]}
    \begin{equation}
        \label{Nawaz2.1ineq}
        \begin{array}{c}
        d(f(i+j), f(i) + f(j))) \le 
        d(\rho_1((f(i+j) + f(i-j),2f(i)))) + \\
        d(\rho_2 (2f(\frac{i+j}{2}),
        f(i)+f(j)))) + \varphi(i,j)
        \end{array}
    \end{equation}
    for all $i,j \in (X,d,\kappa)$.
    Then~$f$ is additive.
\end{assert}

Among the shortcomings and questionable
items in this assertion:
\begin{itemize}
    \item There appear to be multiple
    parentheses errors.
    \item It is not clear if we are to
    consider~$X$ and~$G$ as subsets
    of~$\Z$, or as subsets of 
    some~$\Z^n$ in which the 
    operations are vector addition and
    scalar multiplication.
    \item Must $X$ and $G$ use the
          same adjacency?
    \item $\rho_1$ and $\rho_2$ are
          undefined. The ``proof"
          treats them as scalars.
    \item It appears $X$ and $G$ are
        sets on which addition makes
        sense, but this is not clarified
        in the hypotheses.
    \item The ``proof" uses a
    function~$g$ without defining it.
\end{itemize}

Perhaps most important, the assertion
appears to reduce to triviality, as
we see in the following.

\begin{prop}
    Given the probable interpretation
    of Assertion~\ref{Nawaz2.1}, we
    have $X = \{0\}$.
\end{prop}

\begin{proof}
    By assumption, $0 \in X$. Suppose
    there exists $j \in X \setminus \{0\}$. By assumption,
    $\frac{0+j}{2} = j/2 \in X$, and
    by a simple induction, 
    $0 \neq j/(2^k) \in X$ for all $k \in \N$.
    But this is impossible, since
    as a digital image, $X$ is a subset
    of some~$\Z^n$. The contradiction
    establishes our assertion.
\end{proof}

``Corollary" 2.3 of~\cite{NawazEtal}
repeats several of the errors
of Assertion~\ref{Nawaz2.1} and
similarly reduces to triviality.

``Corollary" 2.5, ``Corollary" 3.2,
and ``Corollary" 3.4
of~\cite{NawazEtal}
rely on the hypotheses of
``Corollary" 2.3 of~\cite{NawazEtal},
hence similarly reduce to triviality.

We conclude that~\cite{NawazEtal}
contributes nothing to our knowledge
of Hyers-Ulam stability in digital
metric spaces.

\section{\cite{ParkEtAl,OkAk}'s contraction mappings}
Let $(X,d)$ be a metric space and let $f: X \to X$. Then $f$
is a {\em Kannan contraction map} if for some
$k$ satisfying $0 < k < 1/2$ and all $x,y \in X$,
\[ d(f(x), f(y)) \le k[d(x,f(x)) + d(y,f(y))].
\]

\subsection{\cite{ParkEtAl}'s Kannan contraction maps}
The paper~\cite{ParkEtAl} was previously discussed
in~\cite{BxSt19}, where some limitations on digital
Kannan contraction mappings were observed; however, an incorrect
proof in~\cite{ParkEtAl} 
was overlooked in~\cite{BxSt19}. 
Here we give a correct proof for
a somewhat more general assertion. 

Stated as Theorem~3.1 of~\cite{ParkEtAl} is the following, which we have paraphrased.

\begin{thm}
\label{ParkEtAl3.1}
    Let $(X,d,\kappa)$ be a digital metric space, where~$d$ is the 
    Euclidean metric. Let 
    $S: X \to X$ be a
    Kannan contraction map.
    Then $S$ has a unique fixed point in $X$.
\end{thm}

This assertion is correct. However,
the argument given for this assertion in~\cite{ParkEtAl}
has an error: it assumes that
its Kannan contraction map is
$\kappa$-continuous, despite the following.

\begin{remarks}
\label{KannanRemark}
    It was shown at Example~4.1
of~\cite{BxSt19} (parts of which we 
restated at our Example~\ref{counterHan5.1})
that a Kannan contraction map
need not be digitally
continuous.
\end{remarks}

We generalize Theorem~\ref{ParkEtAl3.1}
slightly by assuming a uniformly
discrete metric.

\begin{thm}
    \label{ParkEtAl3.1Corrected}
     Let $(X,d)$ be a uniformly discrete digital 
     metric space. Let $S: X \to X$ such that for some
     $\alpha$ satisfying $0 < \alpha < 1/2$ and all $x,y \in X$,
    \[ d(Sx, Sy) \le \alpha [d(x,Sx) + d(y,Sy)]
    \]
   (i.e.,~$S$ is a Kannan
   contraction map).
   Then $S$ has a unique fixed point in $X$.
\end{thm}

\begin{proof}
    We will build on the argument given in~\cite{ParkEtAl}, modifying
    as necessary.

    Let $x_0 \in X$. Consider the sequence $x_{n+1} = Sx_n$, $n \ge 0$.
    Then for $n>1$,
    \[ d(x_n,x_{n+1}) = d(Sx_{n-1},Sx_n) \le 
       \alpha [d(x_{n-1},Sx_{n-1}) + d(x_n,Sx_n)] 
       \]
    \[ = \alpha[d(x_{n-1}, x_n) +  d(x_n, x_{n+1})]~~~\mbox{or}
    \]
    \[ d(x_n,x_{n+1}) \le \frac{\alpha}{1 - \alpha} d(x_{n-1}, x_n).
    \]

    Let $\beta = \frac{\alpha}{1 - \alpha} < 1$. An easy
    induction yields $d(x_n,x_{n+1}) \le \beta^n d(x_0,x_1)$ and
    therefore
    \[ d(x_n,x_{n+k}) \le \mbox{ (by triangle inequality) }
       \sum_{j=1}^k d(x_{n+j-1},x_{n+j}) \le 
       \]
       \[
        \sum_{j=1}^k \beta^{n+j-1} d(x_0,x_1) \le
        \frac{\beta^{n+k}}{1 - \beta} \, d(x_0,x_1) \to_{n \to \infty} 0.
    \]
    Thus, the sequence $x_n$ is a Cauchy sequence.
    
    At this point, \cite{ParkEtAl} reasons 
    incorrectly, claiming 
    $\{x_n\}$ converges to a
    limit~$v$ by virtue of an alleged $(\kappa,\kappa)$-continuity of~$S$.
    We noted in Remark~\ref{KannanRemark} that
    ~$S$ need not be $(\kappa,\kappa)$-continuous. Instead, we reason that
    since $(X,d)$ is uniformly discrete, 
    Proposition~\ref{eventuallyConst} implies that
    $\{x_n\}$ is eventually constant, so 
    $x_n = x_{n+1} = Sx_n$ is a fixed point for almost
    all~$n$.

We argue as in~\cite{ParkEtAl} 
to prove the uniqueness
claim. If $a$ and $b$ are fixed points of $S$, then
\[ d(a,b) = d(Sa,Sb) \le \alpha [d(a,Sa) + d(b,Sb)] = \alpha (0+0) = 0
\]
so $a=b$.
\end{proof}

\subsection{\cite{OkAk}'s citations}
The paper~\cite{OkAk} attempts to prove fixed point
assertions for bicommutative Kannan mappings and for 
functions similar to Reich contraction mappings.
We show that the arguments offered as proofs have many errors; that in some cases,
stronger results exist in the literature or can be proved
briefly by using the 
literature; and that some
important cases reduce 
to triviality.

The paper is further marred by improper citations, including
\[
\begin{array}{lll}
\mbox{\underline{Statement in~\cite{OkAk}}}     & 
\mbox{\underline{\cite{OkAk}'s attribution}} & 
\mbox{\underline{Should be}}  \\
     \mbox{Definition~1.1 - digital image} & \mbox{\cite{KalJain}} & \mbox{\cite{Rosenfeld}} \\
     \mbox{Definition~1.2 - digital metric space} & \mbox{\cite{KalJain}} & \mbox{\cite{EgeKaraca-Ban}} \\
     \mbox{Lipschitz condition} & \mbox{\cite{Indrati}} & \mbox{classic} \\
     \mbox{digital contraction mapping} & \mbox{\cite{RaniEtAl}} & 
     \mbox{\cite{EgeKaraca-Ban}} \\
     \mbox{Theorem~2.1 - digital Kannan contraction} & \mbox{\cite{Adey}} &
     \mbox{\cite{ParkEtAl}}
\end{array}
\]

\subsection{\cite{OkAk}'s Theorem 3.1}
Upon correction for missing parentheses and using abbreviated language,
the following is stated as ``Theorem"~3.1 of~\cite{OkAk}.

\begin{assert}
\label{OkAk3.1}
Let $(M, \rho, \mu)$ be a digital metric space. Suppose
$\Psi, \phi: M \to M$, $0 < g < 1/2$, such that for all
$u,v \in M$,
\begin{equation} 
\label{Kannen-Psi}
\rho(\Psi(u), \Psi(v)) \le g[\rho(u, \Psi(u)) + \rho(v, \Psi(v))]
\end{equation}
\begin{equation}
\label{Kannen-phi}
\rho(\phi(u), \phi(v)) \le g[\rho(u, \phi(u)) + \rho(v, \phi(v))]
\end{equation}
\[ \Psi(u) \subset \phi(u) \mbox{  {\rm [}sic - perhaps this is intended to say 
$\Psi(M) \subset \phi(M)${\rm ]} },
 \]
 and
 $\Psi$ and $\phi$ commute on $M$. Then $\Psi$ and $\phi$
 have a unique common fixed point.
\end{assert}

The ``proof" of this assertion in~\cite{OkAk} has many errors, including
the following.
\begin{itemize}
    \item The 3rd line of the ``proof" claims, with no obvious
    justification, that $\rho(\Psi(u),\Psi(v)) \le \rho(\phi(u), \phi(v))$.
    In subsequent lines, the authors claim to derive this inequality,
    again without obvious justification; a dubious step being the
    unjustified claim that
    \[ g[\rho(u, \Psi(u)) + \rho(v, \Psi(v))] \le
    \rho(\phi(u),\phi(v))
    \]
    \item For $u_0 \in M$, a sequence is defined by
          $\Psi(u_{n+1}) = \phi(u_0)$. It seems likely that the right
          side should have some expression in $u_n$.
    \item The next line of the proof claims $\Psi(\Psi(u_n))=\Psi(u_{n+1})$.
          Perhaps this would follow from a correct definition of
          $u_n$, but as written, the claim is unjustified.
    \item At the last line of page 241, it is claimed that $\Psi$ and
    $\phi$ are $(u,u)$-continuous. No notion of $(u,u)$-continuity has
    been defined. If the authors intend $(\mu,\mu)$ continuity, this is
    unjustified. 
    \begin{itemize}
        \item The authors may be assuming continuity with respect to
              the metric~$\rho$, but that would establish nothing:
              $M \subset \Z^n$ for some $n$ is a uniformly discrete
              topological space with respect to standard metrics, so
              any $f: M \to M$ has such continuity without implying
              a desired convergence.
        \item The authors may be assuming digital continuity, but that
              would be incorrect, as 
              Remark~\ref{KannanRemark}
              notes that a Kannan contraction mapping need not be
              digitally continuous.
    \end{itemize}
\end{itemize}
Confusion propagates from these errors. Thus we must assume that
Assertion~\ref{OkAk3.1} is not proven in~\cite{OkAk}.

\subsection{\cite{OkAk}'s Reich contraction and related maps}
In real analysis, we have
the following.

\begin{definition}
\label{realReich}
    {\rm \cite{Reich}}
    Let $(X,d)$ be a complete
    metric space. Let 
    $T: X \to X$ be a function
    such that for all 
    $x,y \in X$ and for some 
    nonnegative constants 
    $a,b,c$ such that
    $a+b+c < 1$,    \begin{equation}
    \label{ReichOrig}
    d(Tx,Ty) \le ad(x,Tx) +
    bd(y,Ty) + cd(x,y)    
    \end{equation}
    Then $T$ is a
{\em Reich contraction mapping}.
\end{definition}

Unfortunately, the notion of
a Reich contraction was 
imported into digital topology
with a different quantifier for the
coefficients $a,b,c$:

\begin{definition}
\label{ReichDef}
{\rm (See Theorem~3.3 of~\cite{ParkEtAl}.)} If $f: X \to X$ for a digital 
metric space $(X, d, \kappa)$, 
where~$d$ is the Euclidean metric, and
\begin{equation}
\label{ReichIneq}
    d(f(x),f(y)) \le ad(x, f(x)) + bd(y, f(y)) + cd(x, y)
\end{equation}
          for all $x,y \in X$ and all nonnegative $a,b,c$ such that
          $a+b+c < 1$, then $f$ is a {\em digital Reich contraction map}.
\end{definition}
Shortcomings of the latter
version were discussed
in~\cite{BxSt19}. 
In this section:
\begin{itemize}
    \item We show that
the notion of a Reich 
contraction map in the sense of
Definition~\ref{ReichDef}
is trivial (see
Proposition~\ref{ReichTriv}).
\item We show that
\cite{OkAk}'s attempts to
deal with maps that seem
inspired by Definition~\ref{ReichDef}
were not very successful.
\end{itemize}

The following is a modest
improvement on Theorem~4.6
of~\cite{BxSt19}.

\begin{prop}
    \label{ReichTriv}
    Let $(X,d,\kappa)$ be a digital metric space. A
    function $f: X \to X$ is a Reich contraction map 
    in the sense of Definition~\ref{ReichDef} if 
    and only if~$f$ is a constant map.
\end{prop}

\begin{proof}
It is elementary, and left to the reader to show,
that a constant~$f$ is a Reich contraction map 
    in the sense of Definition~\ref{ReichDef}.
    
Since $X$ is finite,~$d$ is uniformly discrete.
We can take $a,b,c$ of~(\ref{ReichIneq})
such that $a=b=0$, giving
\[ d(fx,fy) < c d(x,y)
\]
where~$c$ is arbitrarily small.
Hence $fx=fy$. The assertion follows.
\end{proof}

By contrast, if we apply an obvious application of
Definition~\ref{realReich}
to digital images, we see 
as follows that such a map
need not be constant.

\begin{exl}
    Let $X$ and $f$ be as
    in Example~\ref{counterHan5.1}.
    Then $f$ is a 
    non-constant function
    that, by~(\ref{BxStExlIneq})
    satisfies the
    quantifiers of
    Definition~\ref{realReich}) with
    $a=b=0$,
    $c=2/\sqrt{5}$.
\end{exl}

Stated as Theorem~3.2 of~\cite{OkAk} is the following (paraphrased for
brevity and greater clarity).

\begin{assert}
\label{OkAk3.2}
    Let $(M,\rho,u)$ be a digital metric space. Let
    $\Psi, \phi: M \to M$ be maps
    such that, for all nonnegative
    $a,b,c$ such that $a+b+c < 1$
    and all $u,v \in M$,
    \begin{equation} 
    \label{OK3.2ineq1}
    \rho(\Psi(u),\Psi(v)) \le
    a \rho(\Psi(u),\Psi(\Psi(u))) +
    b \rho(\Psi(v),\Psi(\Psi(v))) +
    c \rho(\Psi(u), \Psi(v))
    \end{equation}
    and
    \[
    \rho(\phi(u),\phi(v)) \le 
    \] 
    \[ a \rho(\phi(u),\phi(\phi(v))) \mbox{ {\em ([sic]} - 
    perhaps the 2nd parameter should be 
    $\phi(\phi(u))$}) 
    \]
   \begin{equation}
    \label{OK3.2ineq2}
     + ~b \rho(\phi(v),\phi(\phi(v))) +
    c \rho(\phi(u), \phi(v))
    \end{equation}
    If $\Psi$ and $\phi$ commute, then they have a common fixed point.
\end{assert}

The argument offered in~\cite{OkAk} has multiple errors,
including the following.
\begin{itemize}
    \item The hypotheses don't
    match the hypotheses of 
    the paper's Theorem~2.3,
    desired by the authors for
    their argument.
    \item The symbol $u$ appears as both the
          name of the adjacency on $M$
          and as a variable for points of~$M$.
    \item In the second line of the ``proof", it is claimed that
    $\Psi(u) \subset \phi(u)$. This makes no sense.
    The authors seem to intend to use 
    the paper's Theorem~2.3, where the same
    error appears, as~\cite{OkAk} misquotes
    Theorem~3.1.4 of~\cite{RaniEtAl}.
    At any rate, Assertion~\ref{OkAk3.2} has 
    no stated containment hypothesis, nor is
    one proven in~\cite{OkAk}.
    \item The next line of the ``proof" claims that
    $\rho(\Psi(u), \Psi(v)) \le \rho(\phi(u), \phi(v))$,
    citing the paper's Theorem~2.3. Since the containment mentioned 
    above, required by a correctly
    quoted version of the paper's
    Theorem~2.3, has not been justified,
    this inequality is not justified.
    \item At the top of page 242, the authors claim to deduce that
    $\Psi(u) \le \phi(u)$. Since we have not assumed that~$M$ is a
    subset of~$\Z$, we are left with the question of what definition
    of ``$\le$" is employed. It seems likely that this statement is
    not what the authors intended.
\end{itemize}

We note this line of investigation reduces
to triviality (whatever
Assertion~\ref{OkAk3.2} might reasonably be
modified to say). Somewhat
similar to Proposition~\ref{ReichTriv},
we have the following.

\begin{prop}
    \label{nonReichTriv}
    Let $(X,d)$ be a 
    metric space. Let
    $F: X \to X$ be a map 
    satisfying ~(\ref{OK3.2ineq1})
    or~(\ref{OK3.2ineq2}). 
   Then~$F$ is a constant function.
\end{prop}

\begin{proof}
    Without loss of generality, $\#X > 1$. 
    
    The inequalities 
    ~(\ref{OK3.2ineq1}) and
    ~(\ref{OK3.2ineq2})
    are hypothesized
    for all nonnegative $a,b,c$ such that
    $a+b+c < 1$. Suppose we take $a=b=0$.
   For $F$ satisfying either of ~(\ref{OK3.2ineq1}) or
    ~(\ref{OK3.2ineq2}),
  for all $u,v \in X$, 
   \[ d(Fu,Fv) \le c d(Fu,Fv)
   \]
Since $c < 1$, $Fu=Fv$, so $F$ is constant.
\end{proof}

\section{\cite{Saluja}'s weakly commutative mappings}
Saluja's \cite{Saluja} gives a common fixed point 
theorem for weakly commutative mappings on a digital
metric space, although its proof is badly written.
In this section, we give a slight generalization
of Saluja's theorem.

\begin{definition}
    {\rm \cite{RaniEtAl}} Let $(X,d,\kappa)$ be a digital metric space.
    Let $S,T: X \to X$ such that
    \begin{equation}
        \label{weaklyCommutIneq}
        d(S(T(x)),T(X(x))) \le d(S(x),T(x)) \mbox{  for all  } x \in X.
    \end{equation}
    Then $S$ and $T$ are {\em weakly commutative} mappings.
\end{definition}

The following is Theorem~3 of~\cite{Saluja}.

\begin{thm}
    \label{SalujaThm3}
    Let $(F,d, Y)$ be a digital metric space, where $d$ is the
    Euclidean metric. Let $J,K: F \to F$ such that
    \begin{itemize}
        \item $J(F) \subset K(F)$;
        \item $K$ is continuous;
        \item $J$ and $K$ are weakly commutative; and
        
    \end{itemize}
    and 
\begin{equation}
\label{SalIneq}
     d(Ju, Jq) \le \xi_1 d(Ku,Kq) + \xi_2 d(Ku,Ju) +
     \xi_3 d(Kq,Jq) \mbox{ for all }
              q,u \in X
\end{equation}
where  $\xi_1, \xi_2, \xi_3 \ge 0$ and
$\xi_1 + \xi_2 + \xi_3 < 1$. Then $J$ and $K$ have a unique common fixed point.
\end{thm}

The proof given in~\cite{Saluja} has some unnecessary 
material such as multiple multi-line arguments
to show assertions of the form
$d(x, x)=0$; and has correct assertions
that seem inadequately explained. Indeed, the
theorem can be generalized as 
Theorem~\ref{improvedSaluja} below, in which we
assume $d$ is a uniformly discrete metric on $X$ 
(a generalization of the Euclidean metric); and we omit
the continuity hypothesis of Theorem~\ref{SalujaThm3}.

\begin{thm}
    \label{improvedSaluja}
     Let $(X,d)$ be a metric space, where $d$ is a
    uniformly discrete metric on $X$. Let $J,K: X \to X$ such that
    $J(F) \subset K(F)$, $J$ and~$K$ are weakly
    commutative, and~(\ref{SalIneq}) holds.
      Then $J$ and $K$ have a unique common fixed point.
\end{thm}

\begin{proof}
    We use ideas of~\cite{Saluja}.

    Let $u_0 \in F$. Since $J(F) \subset K(F)$,
    there exists a sequence $\{u_n\}$ such that
\begin{equation}
\label{salujaThm3Induction}
     Ku_n = Ju_{n-1}.
\end{equation}
Then
\[ d(Ku_{n+1}, Ku_n) = d(Ju_n,Ju_{n-1}) \le 
\mbox{ (by (\ref{SalIneq})) }
\]
\[ \xi_1d(Ku_n,Ku_{n-1}) + \xi_2d(Ku_n, Ju_n) +
    \xi_3d(Ku_{n-1},Ju_{n-1}) =
    \] 
\[ \xi_1d(Ku_n,Ku_{n-1}) + \xi_2d(Ku_n, Ku_{n+1}) +
    \xi_3d(Ku_{n-1},Ku_n) =
    \]
    \[
(\xi_1 + \xi_3) d(Ku_n,Ku_{n-1}) + \xi_2 d(Ku_n,Ku_{n+1})
\]
or
\[ d(Ku_{n+1}, Ku_n) \le
   \frac{\xi_1 + \xi_3}{1 - \xi_2} d(Ku_n,Ku_{n-1}), 
   \mbox{ i.e., }
\]
\[ d(Ku_{n+1}, Ku_n) \le cd(Ku_n,Ku_{n-1}), 
\]
where $c = \frac{\xi_1 + \xi_3}{1 - \xi_2}$. 
From~(\ref{SalIneq}), we have $0 < c < 1$.
Therefore, $\{Ku_n\}$ is a Cauchy sequence. It follows from
Proposition~\ref{eventuallyConst} that for some $a \in X$,
$Ku_n = a$ for almost all~$n$. From~(\ref{salujaThm3Induction})
we have 
\begin{equation}
    \label{JKcoincidence}
    Ju_n = Ku_n = a \mbox{ for almost all } n.
\end{equation}

By (\ref{JKcoincidence}) and~(\ref{weaklyCommutIneq}), 
for almost all~$n$,
\[ d(Ja,Ka) = d(JKu_n,KJu_n) \le
    d(Ju_n, Ku_n) = d(a,a) = 0\]
so
\begin{equation}
\label{aIsCoincidencePt}
    Ja = Ka
\end{equation}
Therefore, (\ref{weaklyCommutIneq}) and~(\ref{aIsCoincidencePt}) imply
\[ d(KJa,JJa) = d(KJa,JKa) \le d(Ka,Ja) = 0
\]
so
\begin{equation}
\label{KJa=JJa}
    KJa = JJa
\end{equation}

Thus
\[ d(Ja,JJa) \le \xi_1 d(Ka,KJa) + \xi_2 d(Ka,Ja) +
   \xi_3 d(KJa,JJa)
\]
\[ = \mbox{ (by (\ref{aIsCoincidencePt}) and (\ref{KJa=JJa})) } ~~\xi_1 d(Ja,JJa) + \xi_2 \cdot 0 + \xi_3 d(JJa,JJa) 
\]
\[ = \xi_1 d(Ja,JKa) + \xi_3 \cdot 0 = \xi_1 d(Ja,JJa).
\]
Thus 
$d(Ja,JJa) = 0$. Thus $Ja$ is a fixed point of~$J$.
By~(\ref{KJa=JJa}), $Ja$ is a common fixed point
of~$J$ and~$K$.

To show the uniqueness of the common fixed point, 
suppose $x_i = Jx_i = Kx_i$, $i \in \{1,2\}$. Then
\[ d(x_1,x_2) = d(Jx_1,Jx_2) \le
   \xi_1 d(Kx_1,Kx_2) + \xi_2 d(Kx_1,Jx_1) +
   \xi_3 d(Kx_2,Jx_2)
\]
\[ = \xi_1 d(x_1,x_2) + 0 + 0
\]
which implies $d(x_1,x_2)=0$, i.e., $x_1 = x_2$.
\end{proof}

\section{\cite{SalJhade22}'s weakly compatible mappings}
The paper~\cite{SalJhade22} claims fixed point
 results for weakly compatible mappings on digital
images. However, its
errors cast doubt on
all of its assertions
that are presented as
original. It also has
several inappropriate
citations.

\subsection{Citation flaws}
Several definitions inappropriately attributed 
in~\cite{SalJhade22}
to~\cite{EgeKaraca-Ban}
should be attributed as
follows.
\[
\begin{array}{ll}
\underline{\mbox{Definition}} & \underline{\mbox{Should be attributed to}} \\
\mbox{2.4 (connected)} &
\cite{Rosenfeld} \\
\mbox{2.5(i), (ii) (continuous)~~~} &
\cite{Rosenfeld,Bx99} \\
\mbox{2.5(iii) (isomorphism)} &
\mbox{\cite{Bx94} (there called ``homeomorphism")} \\
\mbox{2.6 (path)}
& \cite{Rosenfeld}
\end{array}
\]

\subsection{\cite{SalJhade22}'s Proposition~2.15}
\cite{SalJhade22} states
(paraphrased) the following 
as its flawed
Definition~2.14.

\begin{definition}
    \label{weaklyCompat}
    {\rm \cite{Dalal17}}
    Two self mappings $S$ and $T$
    of a digital metric space
    $(X,d,\rho)$ are called 
    {\em weakly compatible} if they
commute at coincidence points; i.e,
if $Sx = Tx$ for all $x \in X$ 
then $STx = TSx$ for all $x \in X$.
\end{definition}

The unfortunate uses of 
``for all $x \in X$" 
imply that $S=T$ - likely not the authors' intention. 
It seems likely
that a better wording would be 
\begin{quote}
  ``... i.e., for all $x \in X$, if $Sx = Tx$  
then $STx = TSx$."  
\end{quote}

At any rate, the following is
stated as Proposition~2.15
of~\cite{SalJhade22}.

\begin{assert}
    \label{SalJhade22-2.15}
    Let $J,K: F \to F$ be two
    weakly compatible maps
on a digital metric space. Suppose~$\eta$
is a unique point of coincidence 
of~$J$ and~$K$, i.e.,
\begin{equation}
\label{weaklyCompatEq}
J(\sigma)=K(\sigma)=\eta \mbox{~~~}
[\mbox{sic}]
\end{equation}
Then $\eta$ is the unique common 
fixed point of~$J$ and~$K$.
\end{assert}

It seems likely that~(\ref{weaklyCompatEq}) is
intended to be
\[ J(\eta)=K(\eta)= \sigma
\]

Regardless, the argument offered
as proof of 
Assertion~\ref{SalJhade22-2.15}
in~\cite{SalJhade22} claims
that~(\ref{weaklyCompatEq}) implies
\[ J(\sigma) = JK(\sigma) =
KJ(\sigma) = K(\sigma). \]
The middle equation
of this chain is due
to the weakly 
compatible assumption.
However, there is no obvious reason
to believe that either the first or
the third equation in this chain
is correct. Therefore,
Assertion~\ref{SalJhade22-2.15}
must be considered unproven.

We note the same
faulty proposition and ``proof"
appear as Proposition~2.15
of~\cite{SalJhade23}, a paper whose
blemishes are discussed
in~\cite{BxBad7}.

\subsection{\cite{SalJhade22}'s
Theorem 3.1}
Stated as Theorem~3.1
of~\cite{SalJhade22}
is the following.
\begin{thm}
    \label{SalJhade3.1}
    Let $(F,\Phi,\Upsilon)$
    be a complete digital
    metric space, where
    $\Phi$ is the 
    Euclidean metric. Let
    $J,K,L,M: F \to F$,
    where $J(F) \subset M(F)$, 
    $K(F) \subset L(F)$, and for some $\xi$
    such that $0 < \xi < 1$,
    \[ \Phi(Ju,Kq) \le
       \xi \Phi(Lu,Mq) 
       \mbox{ for all } u,q \in F.
    \]
 If $L$ and $M$ are continuous mappings 
 and $\{J,L\}$ and $\{K, M\}$ are pairs
of commutative mappings then there exists a unique common fixed point
in $F$ for all four mappings $J, K, L, M$.
\end{thm}

The argument offered as
proof of this assertion
is largely correct, although
marked by fixable typos and
some steps that should have
been better explained. Further,
its hypotheses of completeness and 
continuity are
unnecessary, and the
choice of $\Phi$ can be
generalized. Thus, we have the
following.

\begin{thm}
    \label{improvedSalJhade3.1}
    Let $(F,\Phi,\Upsilon)$
    be a digital
    metric space, where
    $\Phi$ is any uniformly discrete 
    metric. Let
    $J,K,L,M: F \to F$,
    where $J(F) \subset M(F)$, 
    $K(F) \subset L(F)$, and for some $\xi$
    such that $0 < \xi < 1$,
    \begin{equation}
    \label{SJ3.1IneqForImproved}
 \Phi(Ju,Kq) \le
       \xi \Phi(Lu,Mq)
       ~\mbox{ for all } u,q \in F.
    \end{equation}
 If $\{J,L\}$ and $\{K, M\}$ are pairs
of commutative mappings then there exists a unique common fixed point
in $F$ for all four mappings $J, K, L, M$.
\end{thm}

\begin{proof}
    We use ideas of~\cite{SalJhade22}.

    Let $u_0 \in F$.
    Since $J(F) \subset M(F)$, there exists
    $u_1 \in F$ such that
    $Ju_0 = Mu_1$, and
    since $K(F) \subset L(F)$, there exists
    $u_2 \in F$ such that $Lu_2 = Ku_1$.
    Inductively, we
    construct sequences
    $\{u_n\} \subset F$
    and $\{q_n\} \subset F$ such
    that for all~$n \ge 0$,
    \[ q_{2n} = Ju_{2n} = Mu_{2n+1},
    \]
    \[ q_{2n+1} =Ku_{2n+1}= Lu_{2n+2}.
    \]

    By~(\ref{SJ3.1IneqForImproved}),
    \[ \Phi(q_{2n},q_{2n+1}) = \Phi(Ju_{2n},Ku_{2n+1}) \le 
    \xi \Phi(Lu_{2n},Mu_{2n+1}) =
    \xi \Phi(q_{2n-1},q_{2n})    
    \]
    Similarly,
    \[ \Phi(q_{2n+1},q_{2n+2}) \le \xi \Phi(q_{2n},q_{2n+1})
    \]
    Hence for all~$n$,
    \[ \Phi(q_{n+1},q_{n+2}) \le \xi \Phi(q_{n},q_{n+1}) \le \xi^{n+1} \Phi(q_0,q_1)
    \to_{n \to \infty} 0.
    \]
    By Remark~\ref{unifDiscrete} and
    Proposition~\ref{eventuallyConst}, there
    exists $\sigma \in F$ such that for
    almost all~$n$, 
    \begin{equation}
        \label{seqsConverge}
        \sigma = q_n = 
        Ju_{2n} = Ku_{2n+1} = 
        Lu_{2n} = Mu_{2n+1}.
    \end{equation}
    We also have, for almost all~$n$,
    using the commutativity of~$J$ with~$L$,
    \begin{equation}
        \label{Jsig=Lsig}
        J \sigma = JLu_{2n} = LJu_{2n} = L \sigma
    \end{equation}
    Similarly, for almost all $n$,
    \begin{equation}
        \label{Ksig=Msig}
        K\sigma = KMu_{2n+1} =
        MKu_{2n+1} =
        M\sigma
    \end{equation}

By~(\ref{Ksig=Msig}), 
(\ref{seqsConverge}), (\ref{SJ3.1IneqForImproved}),
\[ \Phi(\sigma,M\sigma)
  = \Phi(\sigma,K\sigma)=\Phi(Ju_{2n},Ku_{2n+1}) \le 
  \]
\[  \xi \Phi(Lu_{2n},Mu_{2n+1}) = \xi \Phi(\sigma,\sigma)=0
\]
which implies $\Phi(\sigma,M\sigma) = 0$, i.e., $\sigma = M\sigma$. With~(\ref{Ksig=Msig}), we
have that $\sigma$ is a
common fixed point of~$K$
and~$M$.

(\ref{Jsig=Lsig}), 
(\ref{seqsConverge}), and
(\ref{SJ3.1IneqForImproved}) also give us
\[ \Phi(L\sigma,\sigma)=
\Phi(J\sigma, \sigma) = \Phi(J\sigma,Ku_{2n+1}) \le \xi \Phi(L\sigma,Mu_{2n+1}) =
\xi \Phi(L\sigma, \sigma)
\]
Thus, $\Phi(L\sigma, \sigma)=0$, so
$L\sigma = \sigma$. With~(\ref{Jsig=Lsig}), we
have that $\sigma$
is a common fixed point
of~$J$ and~$L$, and therefore, of
the four maps $J,K,L,M$.

Uniqueness is shown as
follows. Suppose
$\sigma_1$ is a common
fixed point of $J,K,L,M$.
Then
\[ \Phi(\sigma,\sigma_1)
= \Phi(J\sigma,K\sigma_1)
\le \xi \Phi(L\sigma,M\sigma_1)
= \xi \Phi(\sigma,\sigma_1)
\]
so $\Phi(\sigma,\sigma_1) = 0$, i.e., $\sigma = \sigma_1$.
\end{proof}

We note an important case or modification
in which Theorems~\ref{SalJhade3.1} and~\ref{improvedSalJhade3.1}
reduce to triviality. Notice:
\begin{itemize}
    \item We consider the
    case $\Upsilon=c_1$,
    $L=M$ is $c_1$-continuous, and $F$
    is $c_1$-connected.
    \item We do not require the containment and the
    commutativity hypotheses of
    Theorems~\ref{SalJhade3.1} and~\ref{improvedSalJhade3.1}.
\end{itemize}

\begin{prop}
        Let $(F,\Phi,c_1)$
    be a digital
    metric space, where
    $\Phi$ is an $\ell_p$ metric or
    the shortest path metric. Let
    $J,K,L: F \to F$,
    where for some $\xi$
    such that $0 < \xi < 1$,
\begin{equation}
\label{L=M}
     \Phi(Ju,Kq) \le
      \xi \Phi(Lu,Lq)
      \mbox{ for all } u,q \in F
\end{equation} 
holds. Then 
\begin{itemize}
    \item $J=K$ and
    \item if $(F,c_1)$ is connected,
and $L$ is $c_1$-continuous, then $J$
    is a constant function.
\end{itemize}  
\end{prop}

\begin{proof}
When we take $q=u$,~(\ref{L=M}) becomes
\[ \Phi(Ju,Ku) \le \xi \Phi(Lu,Lu) = 0,
\]
so $\Phi(Ju,Ku)=0$,
i.e., $Ju=Ku$. Thus $J=K$.

If $L$ is $c_1$-continuous, then
for $x \adj_{c_1} y$, $Lx \adjeq_{c_1} Ly$,
so
\[ \Phi(Jx,Jy) = \Phi(Jx,Ky)
\le
\xi \Phi(Lx,Ly) \le \xi
< 1
\]
so $Jx=Jy$. If
$F$ is $c_1$-connected,
we have that $J=K$ is a
constant function.
\end{proof}

\subsection{\cite{SalJhade22}'s
Example 3.2}
Example~3.2 of~\cite{SalJhade22}
makes a unique common fixed point
claim for the four functions
\[ J(u)=K(u)=\frac{1}{2},~~~~L(u)= \frac{u+1}{3},~~~~
M(u)=\frac{2u+1}{4}
\]
on the basis of a
previous claim about
digital functions.
But these are not digital
functions, since none of them is
integer-valued.

\subsection{\cite{SalJhade22}'s Theorems 3.4 and 3.5}
Since Assertion~\ref{SalJhade22-2.15}, shown above to be unproven,
is called upon in the arguments
given as proof for the assertions
stated as Theorem~3.4 and 
Theorem~3.5 of~\cite{SalJhade22},
these assertions must be considered
unproven. 

``Theorem" 3.4 of~\cite{SalJhade22} is 
paraphrased as follows.

\begin{assert}
    \label{SJ3.4}
    Let $(X,d,\kappa)$ be a complete digital
    metric space, where~$d$ is the Euclidean
    metric. Let $J,K: X \to X$ such that
    $J(X) \subset K(X)$. Suppose for some
    $\xi$, $0 < \xi < 1$, and all $u,q \in X$,
    \begin{equation} 
    \label{SJ3.4ineq}
    d(Ju,Jq) \le \xi d(Ku, Kq).
    \end{equation}
    Suppose $K(X)$ is complete and $\{J,K\}$
    is a pair of weakly compatible mappings.
    Then there exists a unique common fixed 
    point of~$J$ and~$K$.
\end{assert}

``Theorem" 3.5 of~\cite{SalJhade22} is 
paraphrased as follows.

\begin{assert}
    \label{SJ3.5triv}
    Let $(X,d,\kappa)$ be a complete digital
    metric space, where~$d$ is the Euclidean
    metric. Let $J,K: X \to X$ such that
    $J(X) \subset K(X)$. Suppose for some
    $\xi$, $0 < \xi < 1/2$, and all $u,q \in X$,
    \begin{equation} 
    \label{SJ3.5ineq}
    d(Ju,Jq) \le \xi [d(Ju, Ku) + d(Jq,Kq)].
    \end{equation}
    Suppose $K(X)$ is complete and $\{J,K\}$
    is a pair of weakly compatible mappings.
    Then there exists a unique common fixed 
    point of~$J$ and~$K$.
\end{assert}

\section{\cite{ShaheenEtAl}'s 
iterative algorithms}
The paper~\cite{ShaheenEtAl}
attempts to study iterative
algorithms for fixed points
in digital metric spaces. It is
marred by improper citations,
undefined terms and symbols,
and logical shortcomings.

\subsection{Improper citations}
Definition~1 of~\cite{ShaheenEtAl}
defines what we have called
$c_u$-adjacency for a digital
image $X \subset \Z^n$, attributing the notion 
to~\cite{EgeKaraca15}.
This notion appeared under
different names in~\cite{HanNon,BxHtpyProps}.

Other bad citations of~\cite{ShaheenEtAl}:
\[
\begin{array}{lll}
\mbox{\underline{Item}} &
\mbox{\underline{\cite{ShaheenEtAl} attribution}} & \mbox{\underline{better attribution}} \\
\mbox{Def. 2 - digital image}
& \mbox{\cite{EgeKaraca15}} &
\mbox{\cite{KongRos}} \\
\mbox{Def. 3 - neighbor}
& \mbox{\cite{EgeKaraca15}} &
\mbox{\cite{KongRos}} \\
\mbox{Def. 5 - digital interval} &
\mbox{\cite{EgeKaraca15}} &
\mbox{\cite{Bx94}} \\
\mbox{digitally connected} &
\mbox{\cite{EgeKaraca15}} &
\mbox{\cite{KongRos}} \\
\mbox{Def. 7 - digital metric space} &\mbox{\cite{SrideviEtAl}} &
\mbox{\cite{EgeKaraca-Ban}} 
\end{array}
\]

Also, ``Proposition"~1 
of~\cite{ShaheenEtAl}
(correctly attributed
to~\cite{EgeKaraca-Ban}), claiming digital 
contraction maps are digitally continuous, 
is (as noted in our Example~\ref{counterHan5.1})
incorrect.

\subsection{Undefined terms and
symbols}
There are several terms and
symbols in~\cite{ShaheenEtAl}
that are undefined and for
which no reference is given.

``Definition" 10 is stated as
follows.
\begin{quote}
    {\em Let $(D,g)$ be a complete metric space, $F: D \to D$
    be a self mapping, and
    $t_{n+1} = f(F,t_n)$
    be a} [sic] {\em iterative procedure. Suppose that
    $\Fix(F) \neq \emptyset$,
    the set of fixed point} 
    [sic], {\em and
sequence $t_n$ converges to
$\ell \in \Fix(F)$.}
\end{quote}

Notice that nothing is defined
by this ``definition". Also, 
the notion of {\em iterative
procedure} has neither been 
defined nor referenced.

In ``Definition"~12, the symbol
$f_{F,\alpha_n}$ is undefined.

In statements~(13) and~(14)
of~\cite{ShaheenEtAl}, the
symbol $\Tilde{\mu}$ appears
without any introduction.
By its usage,~$\Tilde{\mu}$
seems to be a function from
$D^2$ to some subset of~$\R$.

\subsection{\cite{ShaheenEtAl}'s Definition 11}
This definition is stated as
follows.
\begin{quote}
    {\em Let $(D,\Tilde{\mu},\rho)$
    be a digital metric space 
    and $\Diamond: D \times [0,1] \to D$ be a mapping such that
    \begin{equation} 
    \label{diamondLinearity}
    \Diamond(\ell, \eta + \mu) = \Diamond(\ell, \eta) 
    + \Diamond(\ell, \mu) \mbox{ and } \Diamond(\ell, 1) = \ell 
    \end{equation}
    We say that a digital metric space have} [sic]
    linear digital structure
    {\em if for all $\ell,m,h,t \in D$ and $\eta,\mu \in [0,1]$,
    if} [sic] {\em it satisfies}
    \begin{equation}
    \label{DiamondIneq}
    \Tilde{\mu}(\Diamond(\ell,\eta) +
    \Diamond(m,\mu),
    \Diamond(h,\eta) +
    \Diamond(t,\mu)) \le
    \eta \Tilde{\mu}(\ell,h) +
      \mu \Tilde{\mu}(m,t).
    \end{equation}
\end{quote}

This ``definition" is flawed
as follows.
\begin{itemize}
\item If we take $\ell = m$
      and $h=t$, the left side
      of~(\ref{DiamondIneq})
      becomes
\[ \Tilde{\mu}(\Diamond(\ell,\eta) +
    \Diamond(\ell,\mu),
    \Diamond(t,\eta) +
    \Diamond(t,\mu)) =
    \Tilde{\mu}(\Diamond(\ell,\eta + \mu), \Diamond(t,\eta + \mu)).
\]
By the hypothesis that the
second parameter of~$\Diamond$
is in~$[0,1]$, we need
another assumption - that
$\eta + \mu \le 1$.
\item Although the second parameter of~$\Diamond$ is required to belong
to~$[0,1]$, the right side of
inequality~(\ref{DiamondIneq})
has terms in which the second
parameters,~$h$ and~$t$, are
assumed in~$D$. 
\item An easy induction argument
based on the first part 
of~(\ref{diamondLinearity})
yields that for $n \in \N$,
\[ n \Diamond(\ell,t) =
   \Diamond(\ell, nt).
\]
If $t > 0$, we can take~$n$
sufficiently large that $nt > 1$, contrary to the requirement
that the second parameter 
of~$\Diamond$ be in~$[0,1]$.
\end{itemize}
We must conclude that
{\em linear digital structure}
is undefined. Since this notion
is woven into ``Theorems" 2, 3,
4, and 5
of~\cite{ShaheenEtAl}, these
assertions are unproven. Thus,
\cite{ShaheenEtAl}
has no assertions
that are all of original, correct,
and correctly proven.

\subsection{\cite{ShaheenEtAl}'s Example 1}
This ``example" assumes
$(D,\Tilde{\mu}, \rho)$ is a
digital metric space,
where $D=N^*$ and
$F: D \to D$ is defined by
\[ F(t) = \frac{t}{2} + 3. 
\]
Note for an odd member of~$D$, $t=2n+1$ for some $n \in \N^*$,
        \[ F(t) = n +3 + 1/2
           \not \in \Z;
        \]
        Thus $F$ is not a 
        valid digital function.

\section{Further remarks}
We paraphrase~\cite{BxBad8}:
\begin{quote}
We have discussed several papers that seek to advance
fixed point assertions for digital metric spaces.
Many of these assertions are incorrect, incorrectly proven, 
or reduce to triviality; consequently, all of these
papers should have been rejected or required to undergo
extensive revisions. This reflects badly not only on
the authors, but also on the referees and editors who
approved their publication.
\end{quote}


\begin{thebibliography}{XX}

\bibitem{Adey}
T. A. Adeyemi, A review of f-contraction mapping on metric spaces. 
Bachelor Degree Project submitted to 
Department of Computer Science and Mathematics, 
Mountain Top University, Nigeria, 2021, 1--54. 
\newline
http://ir.mtu.edu.ng/jspui/handle/123456789/274

\bibitem{Borsuk}
K. Borsuk, On some metrizations of the hyperspace of compact sets, 
{\em Fundamenta Mathematicae} 41 (1954), 168--202
\newline
https://www.impan.pl/en/publishing-house/journals-and-series/fundamenta-mathematicae/all/41/2/94163/on-some-metrizations-of-the-hyperspace-of-compact-sets 

\bibitem{Bx94}
L. Boxer, Digitally continuous functions,
{\em Pattern Recognition Letters}
15 (8) (1994), 833-839
\newline
https://www.sciencedirect.com/science/article/abs/pii/0167865594900124

\bibitem{Bx99}
L. Boxer, A classical construction for the digital fundamental group, 
{\em Journal of Mathematical Imaging and Vision} 10 (1999),
51--62.
\newline
https://link.springer.com/article/10.1023/A$\%$3A1008370600456

\bibitem{BxHtpyProps}
L. Boxer,
Homotopy properties of
sphere-like digital images,
{\em Journal of Mathematical
Imaging and Vision} 24 (2) (2006),
167--175
\newline
https://link.springer.com/article/10.1007/s10851-005-3619-x 

\bibitem{Bx19}
L. Boxer, 
Remarks on Fixed Point Assertions in Digital Topology, 2, 
{\em Applied General Topology} 20, (1) (2019), 155--175.
\newline
https://polipapers.upv.es/index.php/AGT/article/view/10667/11202

\bibitem{Bx19-3}
L. Boxer, Remarks on Fixed Point Assertions in Digital Topology, 3,
{\em Applied General Topology} 20 (2) (2019), 349--361. 
\newline
https://polipapers.upv.es/index.php/AGT/article/view/11117


\bibitem{Bx20}
 L. Boxer, Remarks on Fixed Point Assertions in Digital Topology, 4, 
 {\em Applied General Topology} 21 (2) (2020), 265--284
\newline
https://polipapers.upv.es/index.php/AGT/article/view/13075

\bibitem{Bx22}
L. Boxer, Remarks on Fixed Point Assertions in 
Digital Topology, 5, 
{\em Applied General Topology} 
23 (2) (2022) 437--451
\newline
https://polipapers.upv.es/index.php/AGT/article/view/16655/14995

\bibitem{BxBad6}
L. Boxer, Remarks on Fixed Point Assertions in Digital Topology, 6,
{\em Applied General Topology}
24 (2) (2023), 
281--305
\newline
https://polipapers.upv.es/index.php/AGT/article/view/18996/16097

\bibitem{BxBad7}
L. Boxer,
Remarks on Fixed Point Assertions in Digital Topology, 7,
{\em Applied General Topology} 25 (1) (2024), 97 - 115
\newline
https://polipapers.upv.es/index.php/AGT/article/view/20026

\bibitem{BxBad8}
L. Boxer, Remarks on Fixed Point Assertions in Digital Topology, 8, 
{\em Applied General Topology} 25 (2) (2024), 457-473
\newline
https://polipapers.upv.es/index.php/AGT/article/view/21074/16938

\bibitem{BEKLL}
 L. Boxer, O. Ege, I. Karaca, J. Lopez, and J. Louwsma,
 Digital fixed points, approximate fixed points, 
 and universal functions, 
 {\em Applied General Topology} 17(2), 2016, 159--172
 \newline
 https://polipapers.upv.es/index.php/AGT/article/view/4704/6675

\bibitem{BxSt19}
L. Boxer and P.C. Staecker,
Remarks on fixed point assertions in digital topology,
{\em Applied General Topology} 20 (1) (2019), 135--153.
\newline
https://polipapers.upv.es/index.php/AGT/article/view/10474/11201

\bibitem{ChartTian}
G. Chartrand and S. Tian, 
Distance in digraphs. 
{\em Computers $\&$ Mathematics with Applications}
34 (11) (1997), 15--23.
\newline
https://www.sciencedirect.com/science/article/pii/S0898122197002162


\bibitem{Dalal17}
S. Dalal, Common fixed point results for weakly compatible map in digital metric
spaces, 
{\em Scholars Journal of Physics, 
Mathematics and Statistics} 4 (4) 
(2017) 196--201
\newline
https://www.saspublishers.com/article/2233/

\bibitem{EgeKaraca15}
O. Ege and I. Karaca, 
Digital homotopy fixed point theory, 
{\em Comptes Rendus
Mathematique} 353 (11) (2015), 1029-1033.
\newline
https://www.sciencedirect.com/science/article/pii/S1631073X15001909

\bibitem{EgeKaraca-Ban}
O. Ege and I. Karaca,
Banach fixed point theorem for digital images,
{\em Journal of Nonlinear Science
and Applications} 8 (3) (2015), 
237--245 \newline
https://www.isr-publications.com/jnsa/articles-1797-banach-fixed-point-theorem-for-digital-images

\bibitem{HanNon}
S.E. Han,
Non-product property of the
digital fundamental group,
{\em Information Sciences} 171 (2005), 73--91
\newline
https://www.sciencedirect.com/science/article/abs/pii/S0020025504001008

\bibitem{HanBan}
S.E. Han, 
Banach fixed point theorem from the
viewpoint of digital topology.
{\em Journal of Nonlinear Science and
Applications} 9 (2016), 895--905 \newline
https://www.isr-publications.com/jnsa/articles-1915-banach-fixed-point-theorem-from-the-viewpoint-of-digital-topology

\bibitem{Han19}
S.-E. Han,
Estimation of the complexity of a digital image from 
the viewpoint of fixed point theory,
{\em Applied Mathematics and Computation} 
347 (2019), 236--248
\newline
https://www.sciencedirect.com/science/article/abs/pii/S009630031830941X

\bibitem{Hyers}
D.H. Hyers,
On the stability of the linear
functional equation,
{\em Proceedings of the National
Academy of Sciences of the USA}
27 (1941), 222--224

\bibitem{Indrati}
C. R. Indrati,
Fixed-point theorems involving Lipschitz in the small,
{\em Abstract and Applied Analysis} 2023, Article ID 5236150
\newline
https://doi.org/10.1155/2023/5236150

\bibitem{KalJain}
R. Kalaiarasia and R. Jain,
Fixed point theory in digital topology,
{\em International Journal of Nonlinear Analysis and Applications}
13 (2021), 157--163
\newline
https://ijnaa.semnan.ac.ir/article\_6375\_8c209dc071d7138bca7163062cd3effe.pdf

\bibitem{Khalimsky}
E. Khalimsky,
Motion, deformation, and homotopy in finite spaces,
{\em Proc. IEEE Intl. Conf. Systems, Man, Cybernetics}
(1987), 227--234

\bibitem{KongRos}
T.Y. Kong and A. Rosenfeld,
Digital topology: introduction and survey,
{\em Computer Vision, Graphics, and Image Processing} 48 (1989), 
357--393
\newline
https://www.sciencedirect.com/science/article/abs/pii/0734189X89901473 

\bibitem{Nadler}
S. B. Nadler, Jr., {\em 
Hyperspaces of Sets}, Marcel Dekker, New York, 1978.

\bibitem{NawazEtal}
S. Nawaz, M.K Hassani, A. Batool,
and A. Akg\"{u}l,
Stability of functional inequality in 
digital metric space,
{\em Journal of Inequalities and 
Applications} 2024, Article 111
\newline
https://doi.org/10.1186/s13660-024-03179-1

\bibitem{OkAk}
A. Okwegye and A. S. Akyenyi,
Bi-commutative digital contraction mapping
and fixed point theorem on digital image and metric spaces,
{\em Dutse Journal of Pure and Applied Sciences} 9 (3a) (2023),
237--245
\newline
https://www.ajol.info/index.php/dujopas/article/view/256861

\bibitem{ParkEtAl}
C. Park, O. Ege, S. Kumar, D. Jain, and J. R. Lee,
Fixed point theorems for various contraction conditions
in digital metric spaces,
Journal of Computational Analysis And Applications 26 (8) (2019),
1451--1458
\newline
https://www.eudoxuspress.com/index.php/journal/issue/view/9/52


\bibitem{RaniEtAl}
A. Rani, K. Jyoti, and A. Rani,
Common fixed point theorems in digital metric spaces
{\em International Journal of Scientific \& Engineering Research}
7 (12) (2016), 1704--1716
\newline
https://www.ijser.org/researchpaper/Common-fixed-point-theorems-in-digital-metric-spaces.pdf

\bibitem{Reich}
S. Reich,
Some remarks concerning
contraction mappings,
{\em Canadian Mathematical Bulletin}
14 (1) (1971), 121--124
\newline
https://www.cambridge.org/core/services/aop-cambridge-core/content/view/62ED0CC002E8224C486ABE631A46D721/S000843950005801Xa.pdf

\bibitem{Rosenfeld}
A. Rosenfeld, `Continuous' functions on digital pictures, 
{\em Pattern Recognition Letters} 4, 1986, 177--184
\newline
https://www.sciencedirect.com/science/article/pii/01678655

\bibitem{Saluja}
A. S. Saluja,
Weakly commutative mappings and common fixed-point theorem in digital metric space,
{\em International Journal of Science and Research} 12 (5) (2023),
1869--1871
\newline
https://www.ijsr.net/getabstract.php?paperid=SR23522104238

\bibitem{SalJhade22}
A.S. Saluja and J. Jhade,
Common fixed point theorems for commutative and weakly 
compatible mappings in digital metric space,
{\em Journal of Hyperstructures} 11 (2) (2022), 280--291
\newline
https://jhs.uma.ac.ir/article\_2581\_1e567825afcfb88574939f88ff84a59a.pdf

\bibitem{SalJhade23}
A.S. Saluja and J. Jhade,
Common fixed point theorem for 
weakly compatible mappings in digital
metric space,
{\em International Journal of 
Mathematics and Computer Research}
11 (5) (2023), 
3448--3450
\newline
https://ijmcr.in/index.php/ijmcr/article/view/566

\bibitem{ShaheenEtAl}
A. Shaheen, A. Batool, A. Ali,
H.A. Sulami, and A. Hussain,
Recent developments in
iterative algorithms for
digital metrics,
{\em Symmetry} 16 (2024), 368 
\newline
https://www.mdpi.com/2073-8994/16/3/368

\bibitem{SrideviEtAl}
K. Sridevi, M.V.R. Kameswari, and D.M.K. Kiran,
Fixed point theorems for
digital contractive type
mappings in digital metric
spaces, 
{\em International Journal of Mathematics Trends and Technology} 
48 (3) (2017), 159--167
\newline
https://ijmttjournal.org/archive/ijmtt-v48p522

\bibitem{Ulam}
S.M. Ulam,
{\em Problems in Modern Mathematics},
Wiley, New York, 1960
\newline
https://www.amazon.com/Problems-Modern-Mathematics-Phoenix-Editions/dp/0486495833

\end{thebibliography}
\end{document}